\documentclass[11pt]{article}
\usepackage{amsmath,amssymb,amsthm}
\usepackage{epsfig}
\usepackage{color}
\usepackage{enumerate}
\usepackage[ruled,vlined,linesnumbered]{algorithm2e}
\usepackage{hyperref}
\usepackage{enumerate}
\usepackage{graphicx}
\usepackage{tikz}
\usepackage{tkz-graph}
\usepackage{tkz-berge}
\usetikzlibrary{arrows.meta,shapes.arrows}
\hypersetup{colorlinks=true, linkcolor=blue, citecolor=blue,
urlcolor=blue}
\usetikzlibrary{calc}

   \newcommand{\bb}[1]{\textcolor{red}{#1}}

\hypersetup{colorlinks=true}

\hypersetup{colorlinks=true, linkcolor=blue, citecolor=blue,urlcolor=blue}

\newcommand{\roeo}{\rho_{e}^{o}} %

\title{Edge open packing: complexity, algorithmic aspects, and bounds}
\date{}
\author{{Bo\v{s}tjan Bre\v{s}ar$^{1,2}$ and Babak Samadi$^{2}$}\vspace{1.25mm}\\
$^{1}$Faculty of Natural Sciences and Mathematics, University of Maribor, Slovenia\\
$^{2}$Institute of Mathematics, Physics and Mechanics, Ljubljana, Slovenia\vspace{1mm}\\
{\tt bostjan.bresar@um.si}\\
{\tt babak.samadi@imfm.si}}
\date{}

\newtheorem{theorem}{Theorem}[section]

\newtheorem{lemma}[theorem]{Lemma}

\newtheorem{obs}[theorem]{Observation}

\theoremstyle{definition}

\theoremstyle{remark}

\begin{document}

\maketitle

\begin{abstract}
Given a graph $G$, two edges $e_{1},e_{2}\in E(G)$ are said to have a common edge $e$ if $e$ joins an endvertex of $e_{1}$ to an endvertex of $e_{2}$. A subset $B\subseteq E(G)$ is an edge open packing set in $G$ if no two edges of $B$ have a common edge in $G$, and the maximum cardinality of such a set in $G$ is called the edge open packing number, $\rho_{e}^{o}(G)$, of $G$. In this paper, we prove that the decision version of the edge open packing number is NP-complete even when restricted to graphs with universal vertices, Eulerian bipartite graphs, and planar graphs with maximum degree $4$, respectively. In contrast, we present a linear-time algorithm that computes the edge open packing number of a tree. We also resolve two problems posed in the seminal paper [Edge open packing sets in graphs, RAIRO-Oper.\ Res.\ 56 (2022) 3765--3776]. Notably, we characterize the graphs $G$ that attain the upper bound $\rho_e^o(G)\le |E(G)|/\delta(G)$, and provide lower and upper bounds for the edge-deleted subgraph of a graph and establish the corresponding realization result. 
\end{abstract}



\section{Introduction and preliminaries} 

Throughout this paper, we consider $G$ as a finite simple graph with vertex set $V(G)$ and edge set $E(G)$. We use~\cite{West} as a reference for terminology and notation which are not explicitly defined here. The {\em (open)} {\em neighborhood} of a vertex $v$ is denoted by $N(v)$, and its {\em closed neighborhood} is $N[v]=N(v)\cup \{v\}$. The {\em minimum} and {\em maximum degrees} of $G$ are denoted by $\delta(G)$ and $\Delta(G)$, respectively. Given the subsets $A,B\subseteq V(G)$, by $[A,B]$ we denote the set of edges with one endvertex in $A$ and the other in $B$. If $X\subseteq V(G)$, then $G[X]$ denotes the subgraph of $G$ induced by the vertices of $X$. On the other hand, if $Y\subseteq E(G)$, then $G[Y]$ denotes the subgraph of $G$ induced by the endvertices of edges in $Y$.

An \textit{induced matching} $M$ of $G$ is a subset $M\subseteq E(G)$ such that $M$ is a matching in which no two edges are joined by an edge of $G$. The \textit{induced matching number} is the maximum cardinality of an induced matching in $G$. The problem of finding a maximum induced matching was originally introduced by Stockmeyer and Vazirani~\cite{sv} (under the name ``risk-free marriage problem") and extensively investigated in literature (see, for instance, \cite{Cam,ddl,Lozin} and references therein). Unlike for the standard matching problem which is solvable in polynomial time, it was proved in~\cite{sv} that the decision version of the induced matching number is an NP-complete problem. In contrast, a linear-time algorithm for determining the induced matching number in trees was found in~\cite{zito}. 

Chelladurai et al.~\cite{cks} introduced the concept of edge open packing in graphs. A subset $B\subseteq E(G)$ is an \textit{edge open packing set} ({\em EOP set} for short) if for each pair of edges $e_{1},e_{2}\in B$, there is no edge from $E(G)$ joining an endvertex of $e_{1}$ to an endvertex of $e_{2}$. In other words, for any two edges $e_{1},e_{2}\in B$, there is no edge $e\in E(G)$ leading to the existence of a path $e_{1}ee_{2}$ or a triangle $e_{1}ee_{2}e_{1}$ in $G$. The \textit{edge open packing number} ({\em EOP number}), $\rho_{e}^{o}(G)$, of $G$  is the maximum cardinality of an EOP set in $G$. We observe that if $B$ is an EOP set (resp. induced matching) of $G$, then $G[B]$ is a disjoint union of induced stars (resp. $K_{1,1}$-stars). In this sense, the problem of induced matching can be considered as a variation of the EOP problem. Note that an EOP set is not necessarily a matching, let alone an induced matching.

An {\em injective $k$-edge coloring} of a graph $G$ is an assignment of colors (integers) from $\{1,\ldots,k\}$ to the edges of $G$ such that any two edges $e_1,e_2\in E(G)$ having a common edge (that is an edge $e$ with one endvertex common with $e_1$ and the other endvertex common with $e_2$) receive different colors. The minimum number $k$ of colors required for an injective edge coloring is called the \textit{injective chromatic index} of $G$, denoted by $\chi'_{i}(G)$. The study of this concept was initiated by Cardoso et al.\ in~\cite{cccd} and investigated in~\cite{fhl,msy} among others. We observe that the color classes of an injective edge coloring of a graph $G$ are the very EOP sets of $G$. In fact, the problem of injective edge coloring amounts to the problem of edge partitioning of a graph into EOP sets. Some real-world applications of these two concepts are given in~\cite{cccd} and~\cite{cks}, respectively. In the only paper so far concerned with EOP sets and numbers~\cite{cks}, the complexity and algorithmic issues have not been considered. In this paper, we fill this gap by proving three different NP-completeness results in specific families of graphs and constructing a linear-time algorithm resolving the problem in trees. 

Section~\ref{sec:comput} of this paper pertains to the study of computational complexity of the EOP problem. We prove that the decision problem associated with $\rho_{e}^{o}$ is NP-complete even when restricted to three special families of graphs, namely, graphs with universal vertices, Eulerian bipartite graphs, and planar graphs with maximum degree $4$.  In~Section~\ref{sec:treealg}, we prove that the EOP number $\rho_{e}^{o}(T)$ and an optimal EOP set for any tree $T$ can be computed/constructed in linear time by exhibiting an efficient algorithm for EOP in trees. A version of dynamic programming is used, where four auxiliary parameters, versions of the EOP number, are being computed. 
Finally, in Section $4$, we solve two problems posed in~\cite{cks}. Firstly, we provide a structural characterization of all graphs $G$ attaining the previously proved upper bound $\rho_{e}^{o}(G)\leq|E(G)|/\delta(G)$. Secondly, we analyze the effect of edge removal on the EOP number of a graph $G$, obtaining a lower and an upper bound for $\roeo(G-e)$, where $e$ is an edge of $G$. We also show that all possible values between the bounds can be realized by a graph and some of its edges.

By a $\rho_{e}^{o}(G)$-set and an $\alpha(G)$-set we represent an EOP set and an independent set of $G$ of cardinality $\rho_{e}^{o}(G)$ and $\alpha(G)$, respectively, where $\alpha(G)$ stands for the independence number of $G$.


\section{Computational complexity}
\label{sec:comput}
We discuss the problem of deciding whether for a given graph $G$ and a positive integer $k$ the graph $G$ has the EOP number at least $k$. We state it as the following decision problem.

\begin{equation}\label{Uni}
\begin{tabular}{|l|}
\hline
\mbox{{\sc Edge open packing problem}} \\
\mbox{{\sc Instance}: A graph $G$ and an integer $k\leq|E(G)|$.}\\
\mbox{{\sc Question}: Is $\rho_{e}^{o}(G)\geq k$?}\\
\hline
\end{tabular}
\end{equation}

In order to study the complexity of the problem (\ref{Uni}), we make use of the {\sc Independent Set Problem} which is known to be NP-complete~\cite{gj}. 

$$\begin{tabular}{|l|}
\hline
\mbox{\sc Independent Set Problem}\\
\mbox{{\sc Instance}: A graph $G$ and an integer $j\leq|V(G)|$.}\\
\mbox{{\sc Question}: Is $\alpha(G)\geq j$?}\\
\hline
\end{tabular}$$\vspace{0.25mm}

As an immediate consequence of the following theorem we deduce that the problem (\ref{Uni}) is NP-complete for graphs of diameter $2$.

\begin{theorem}\label{NP-complete}
{\sc Edge Open Packing Problem} is NP-complete even for graphs with universal vertices.
\end{theorem}
\begin{proof}
The problem clearly belongs to NP because checking that a given subset of edges is indeed an EOP set of cardinality at least $k$ can be done in polynomial time.

Suppose that $G$ is a graph of order $n$ and size $m$. We set $k=m+n+j$. Let $H$ be obtained from $G$ by adding the new vertices $v,u_{1},u_{2},\dots,u_{m+n}$ and joining $v$ to all vertices in $V(G)\cup \{u_{1},u_{2},\dots,u_{m+n}\}$. It is easy to observe that the construction of the graph $H$ can be accomplished in polynomial time and that $H$ has the universal vertex $v$. Let $I$ be an $\alpha(G)$-set. Since $J=I\cup \{u_{1},u_{2},\dots,u_{m+n}\}$ is an independent set in $H$, it is easily observed that $B=\{vw\in E(H):\, w\in J\}$ is an EOP in $H$. Therefore, $\rho_{e}^{o}(H)\geq|B|=m+n+\alpha(G)$.

Let $B'$ be a $\rho_{e}^{o}(H)$-set. Suppose first that $B'\cap E(G)\neq \emptyset$ and that $xy\in B'\cap E(G)$. If there exists an edge $vu_{r}\in B'$ for some $1\leq r\leq m+n$, then we have the path $u_{r}vxy$ in which $u_{r}v,xy\in B'$, which is impossible. Therefore, $B'\cap \{vu_{i}:\, 1\leq i\leq m+n\}=\emptyset$. This, in particular, implies that $|B'|\leq|E(H)\setminus \{vu_{i}:\, 1\leq i\leq m+n\}|=m+n<m+n+\alpha(G)$, a contradiction. Hence, $B'\cap E(G)=\emptyset$.

Suppose that $|B'\cap \{vw\in E(H):\, w\in V(G)\}|>\alpha(G)$ and $vw_{1},\dots,vw_{\alpha(G)+1}\in B'\cap \{vw\in E(H):\, w\in V(G)\}$. Clearly, there exists an edge $w_{r}w_{s}\in E(G)$ for some $1\leq r\neq s\leq \alpha(G)+1$. This, in turn, leads to the triangle $vw_{r}w_{s}v$. So, the edges $vw_{r}$ and $vw_{s}$ have the common edge $w_{r}w_{s}$, which is a contradiction. This implies that $|B'\cap \{vw\in E(H):\, w\in V(G)\}|\leq \alpha(G)$, and hence $\rho_{e}^{o}(H)=|B'|\leq m+n+\alpha(G)$. This shows that $\rho_{e}^{o}(H)=m+n+\alpha(G)$.

Our reduction is now completed by taking into account the fact that $\rho_{e}^{o}(H)\geq k$ if and only if $\alpha(G)\geq j$. Since {\sc Independent Set Problem} is NP-complete, we have the same for {\sc Edge Open Packing Problem} restricted to the family of graphs with universal vertices.
\end{proof}

In spite of the fact that {\sc Independent Set Problem} is solvable in polynomial time for bipartite graphs, {\sc Edge open packing problem} remains NP-complete for a subfamily of bipartite graphs. In this sense, one might even say that it is harder than {\sc Independent Set Problem}.

\begin{theorem}\label{Eulerian-Bipartite}
{\sc Edge Open Packing Problem} is NP-complete even for Eulerian bipartite graphs.
\end{theorem}
\begin{proof}
We need to prove that the problem is NP-hard as we already proved that it is in NP. Let $G$ be any graph with vertex set $\{v_{1},\ldots,v_{n}\}$ and size $m$. We construct the graph $H_G$ derived from $G$ as follows. For each vertex $v_{i}\in V(G)$, take two vertices $v_i$ and $v_i'$ and connect them by an edge, and for each edge $v_iv_j\in E(G)$ add six paths, three of which connecting $v_i$ and $v_j'$ and the other three connecting $v_i'$ and $v_j$ as follows:
\begin{center}
$P_{ij}^1:v_{i}x_{ij}\widetilde{x_{ij}}x_{ji}''x_{ji}'v_{j}'$, $P_{ij}^2:v_{i}y_{ij}\widetilde{y_{ij}}y_{ji}''y_{ji}'v_{j}'$, $P_{ij}^1:v_{i}z_{ij}\widetilde{z_{ij}}z_{ji}''z_{ji}'v_{j}'$
\end{center}
\begin{center}
$P_{ji}^1:v_{i}'x_{ij}'x_{ij}''\widetilde{x_{ji}}x_{ji}v_{j}$,
$P_{ji}^2:v_{i}'y_{ij}'y_{ij}''\widetilde{y_{ji}}y_{ji}v_{j}$,
$P_{ji}^3:v_{i}'z_{ij}'z_{ij}''\widetilde{z_{ji}}z_{ji}v_{j}$.
\end{center}
It is clear that $|V(H_G)|=2n+24m$, $|E(H_G)|=n+30m$ and that the construction of $H_G$ can be accomplished in polynomial time. Moreover, it is evident by the structure that $H_{G}$ is bipartite (a local induced subgraph of the graph $H_{G}$ is depicted in Figure~\ref{fig1}). 

Let $I$ be an $\alpha(G)$-set. If one of the endvertices of $v_{i}v_{j}\in E(G)$, say $v_{j}$, belongs to $I$, then let 
\begin{equation*}
\label{eq:firstAij}
A_{ij}=\{v_{j}v_{j'}\}\cup \{v_{j}x_{ji},v_{j}y_{ji},v_{j}z_{ji},x_{ij}''x_{ij}',y_{ij}''y_{ij}',z_{ij}''z_{ij}'\} \cup \{x_{ij}\widetilde{x_{ij}},\widetilde{x_{ij}}x_{ji}'',y_{ij}\widetilde{y_{ij}},\widetilde{y_{ij}}y_{ji}'', z_{ij}\widetilde{z_{ij}},\widetilde{z_{ij}}z_{ji}''\}.
\end{equation*}
Otherwise, let
\begin{equation*}
\label{eq:secondAij}
A_{ij}= \{x_{ij}\widetilde{x_{ij}},\widetilde{x_{ij}}x_{ji}'',y_{ij}\widetilde{y_{ij}},\widetilde{y_{ij}}y_{ji}'',z_{ij}\widetilde{z_{ij}},\widetilde{z_{ij}}z_{ji}''\}\cup\{x_{ij}'x_{ij}'',x_{ij}''\widetilde{x_{ji}},y_{ij}'y_{ij}'',y_{ij}''\widetilde{y_{ji}}, z_{ij}'z_{ij}'',z_{ij}''\widetilde{z_{ji}}\}.
\end{equation*}
It is then readily checked that $\bigcup_{v_{i}v_{j}\in E(G)}A_{ij}$ is an EOP set of $H_{G}$, and hence 
\begin{equation}\label{eq0}
\rho_{e}^{o}(H_{G})\geq|\cup_{v_{i}v_{j}\in E(G)}A_{ij}|=12m+\alpha(G).
\end{equation}

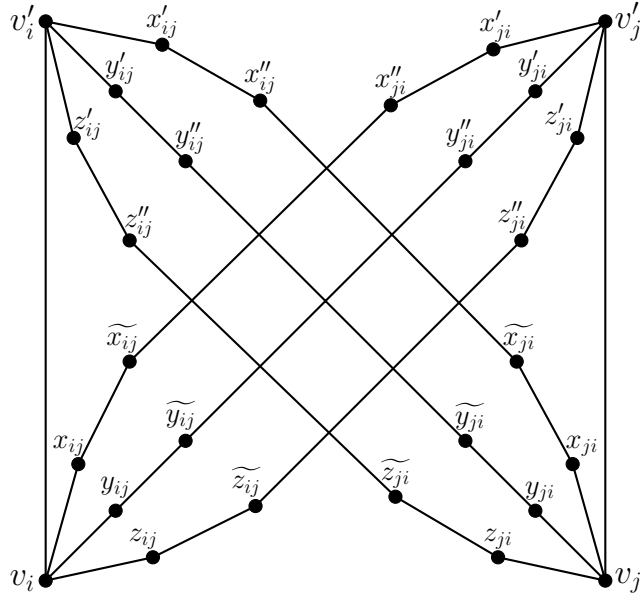
\begin{figure}[ht]
 \centering
\begin{tikzpicture}[scale=0.62, transform shape]

\draw[fill=black] (0,0) circle (4pt); 
\draw[fill=black] (0,-12) circle (4pt); 
\draw[fill=black] (12,0) circle (4pt); 
\draw[fill=black] (12,-12) circle (4pt); 

\node at (-0.5,0) {{\huge $v_{i}'$}};
\node at (-0.5,-12) {{\huge $v_{i}$}};
\node at (12.5,0) {{\huge $v_{j}'$}};
\node at (12.5,-12) {{\huge $v_{j}$}};

\draw[thick] (0,0)--(0,-12)--(12,0)--(12,-12)--(0,0);

\draw[fill=black] (1.5,-10.5) circle (4pt); 
\draw[fill=black] (3,-9) circle (4pt); 
\draw[fill=black] (9,-3) circle (4pt); 
\draw[fill=black] (10.5,-1.5) circle (4pt); 

\node at (1.5,-10) {{\LARGE $y_{ij}$}};
\node at (2.9,-8.45) {{\LARGE $\widetilde{y_{ij}}$}};
\node at (8.9,-2.45) {{\LARGE $y_{ji}''$}};
\node at (10.45,-0.95) {{\LARGE $y_{ji}'$}};

\draw[fill=black] (1.5,-1.5) circle (4pt); 
\draw[fill=black] (3,-3) circle (4pt); 
\draw[fill=black] (9,-9) circle (4pt); 
\draw[fill=black] (10.5,-10.5) circle (4pt); 

\node at (1.6,-1) {{\LARGE $y_{ij}'$}};
\node at (3.1,-2.5) {{\LARGE $y_{ij}''$}};
\node at (9.1,-8.5) {{\LARGE $\widetilde{y_{ji}}$}};
\node at (10.6,-10.1) {{\LARGE $y_{ji}$}};


\draw[fill=black] (0.7,-9.5) circle (4pt); 
\draw[fill=black] (1.8,-7.3) circle (4pt); 
\draw[fill=black] (7.4,-1.8) circle (4pt); 
\draw[fill=black] (9.6,-0.6) circle (4pt); 

\node at (0.5,-9.1) {{\LARGE $x_{ij}$}};
\node at (1.65,-6.8) {{\LARGE $\widetilde{x_{ij}}$}};
\node at (7.4,-1.3) {{\LARGE $x_{ji}''$}};
\node at (9.65,-0.1) {{\LARGE $x_{ji}'$}};

\draw[thick] (0,-12)--(0.7,-9.5)--(1.8,-7.3)--(7.4,-1.8)--(9.6,-0.6)--(12,0);

\draw[fill=black] (2.3,-11.5) circle (4pt); 
\draw[fill=black] (4.5,-10.4) circle (4pt); 
\draw[fill=black] (10.2,-4.7) circle (4pt); 
\draw[fill=black] (11.4,-2.5) circle (4pt); 

\node at (2.1,-11.1) {{\LARGE $z_{ij}$}};
\node at (4.3,-9.88) {{\LARGE $\widetilde{z_{ij}}$}};
\node at (10,-4.2) {{\LARGE $z_{ji}''$}};
\node at (11,-2.1) {{\LARGE $z_{ji}'$}};

\draw[thick] (0,-12)--(2.3,-11.5)--(4.5,-10.4)--(10.2,-4.7)--(11.4,-2.5)--(12,0);


\draw[fill=black] (2.5,-0.5) circle (4pt); 
\draw[fill=black] (4.6,-1.7) circle (4pt); 
\draw[fill=black] (10.1,-7.3) circle (4pt); 
\draw[fill=black] (11.3,-9.5) circle (4pt); 

\node at (2.5,0) {{\LARGE $x_{ij}'$}};
\node at (4.6,-1.2) {{\LARGE $x_{ij}''$}};
\node at (10.15,-6.8) {{\LARGE $\widetilde{x_{ji}}$}};
\node at (11.5,-9.1) {{\LARGE $x_{ji}$}};

\draw[thick] (0,0)--(2.5,-0.5)--(4.6,-1.7)--(10.1,-7.3)--(11.3,-9.5)--(12,-12);

\draw[fill=black] (9.7,-11.5) circle (4pt); 
\draw[fill=black] (7.5,-10.2) circle (4pt); 
\draw[fill=black] (1.8,-4.7) circle (4pt); 
\draw[fill=black] (0.6,-2.5) circle (4pt); 

\node at (9.7,-11.1) {{\LARGE $z_{ji}$}};
\node at (7.5,-9.7) {{\LARGE $\widetilde{z_{ji}}$}};
\node at (2,-4.3) {{\LARGE $z_{ij}''$}};
\node at (0.9,-2.2) {{\LARGE $z_{ij}'$}};

\draw[thick] (0,0)--(0.6,-2.5)--(1.8,-4.7)--(7.5,-10.2)--(9.7,-11.5)--(12,-12);

\end{tikzpicture}
\caption{{\small The subgraph of $H_{G}$ corresponding to the adjacent vertices $v_{i}$ and $v_{j}$ of $G$.}}\label{fig1}
\end{figure}

For the proof of the reversed inequality, let $B$ be a $\rho_{e}^{o}(H_{G})$-set. We have
\begin{equation}\label{eq1}
\rho_{e}^{o}(H_{G})=|B|=\sum_{v_{i}v_{j}\in E(G)}\Big{|}B\cap\Big{(}\bigcup_{k=1}^{3}\big{(}E(P_{ij}^{k})\cup E(P_{ji}^{k})\big{)}\Big{)}\Big{|}+\Big{|}B\cap \{v_{i}v_{i}':\, v_{i}\in V(G)\}\Big{|}.
\end{equation}
By the structure of $H_{G}$ and the definition of an EOP set, we observe that $\big{|}B\cap\big{(}\bigcup_{k=1}^{3}(E(P_{ij}^{k})\cup E(P_{ji}^{k}))\big{)}\big{|}$ does not exceed $12$ for each $v_{i}v_{j}\in E(G)$. To see this, note that the vertices of any two distinct paths $P_{ij}^k$ (as well as those of any two distinct paths $P_{ij}^{\ell}$) induce the cycle $C_{10}$, for which it is easily seen that $\rho_e^o(C_{10})=4$. In addition, adding the third path between $v_i$ and $v_j'$, obtaining the subgraph induced by $\cup_{k=1}^{3}V(P_{ij}^k)$, we can only add two more edges to an EOP set of this subgraph, which yields the above-mentioned upper bound $12$.

Next, we define three subsets of $E(G)$, namely 
$$E_{k}=\{v_{i}v_{j}\in E(G):\ |\{v_{i}v_{i}',v_{j}v_{j}'\}\cap B|=k\}$$
for $k\in\{0,1,2\}$. 
Clearly, $E_0$, $E_1$ and $E_2$ are pairwise disjoint, and $E(G)=E_{0}\cup E_{1}\cup E_{2}$. We then, in view of (\ref{eq1}), have
\begin{equation}\label{Eu2}
\begin{array}{lcl}
\rho_{e}^{o}(H_{G})=|B|&=&\sum\limits_{\substack{v_{i}v_{j}\in E_0}}\Big{|}B\cap\Big{(}\bigcup_{k=1}^{3}\big{(}E(P_{ij}^{k})\cup E(P_{ji}^{k})\big{)}\Big{)}\Big{|}\\
&+&\sum\limits_{v_{i}v_{j}\in E_1}\Big{|}B\cap\Big{(}\bigcup_{k=1}^{3}\big{(}E(P_{ij}^{k})\cup E(P_{ji}^{k})\big{)}\Big{)}\Big{|}\\
&+&\sum\limits_{v_{i}v_{j}\in E_2}\Big{|}B\cap\Big{(}\bigcup_{k=1}^{3}\big{(}E(P_{ij}^{k})\cup E(P_{ji}^{k})\big{)}\Big{)}\Big{|}+\Big{|}B\cap \{v_{i}v_{i}':\, v_{i}\in V(G)\}\Big{|}.
\end{array}
\end{equation}

It is not difficult to verify (see also Figure~\ref{fig1}) that $\big{|}B\cap\big{(}\bigcup_{k=1}^{3}\big{(}E(P_{ij}^{k})\cup E(P_{ji}^{k})\big{)}\big{)}\big{|}\leq12$ if $v_{i}v_{j}\in E_{0}\cup E_{1}$, and $\big{|}B\cap\big{(}\bigcup_{k=1}^{3}\big{(}E(P_{ij}^{k})\cup E(P_{ji}^{k})\big{)}\big{)}\big{|}\leq9$ if $v_{i}v_{j}\in E_{2}$. So, in view of (\ref{Eu2}), we have
\begin{equation}\label{Eu3}
\rho_{e}^{o}(H_{G})=|B|\leq12|E_{0}|+12|E_{1}|+9|E_{2}|+|B\cap \{v_{i}v_{i}':\, v_{i}\in V(G)\}|.
\end{equation}

Let $W=V(G[E_2])$. Note that $G[E_2]$ has no isolated vertices if $W\ne \emptyset$, and so $2|E_2|\ge |W|$. On the other hand, let $R=\{v_i\in V(G):\, v_iv_i'\in B\}\setminus W$. Note that $R$ is an independent set in $G$. Clearly, $|B\cap \{v_{i}v_{i}':\, v_{i}\in V(G)\}|$=$|W|+|R|$. Inserting this to~\eqref{Eu3}, we get
$$\rho_{e}^{o}(H_{G})=|B|\leq 12|E_{0}|+12|E_{1}|+9|E_{2}|+|W|+|R|\le 12|E_{0}|+12|E_{1}|+11|E_{2}|+\alpha(G),$$
which is strictly less than $12m+\alpha(G)$ unless $E_2=\emptyset$. Due to~\eqref{eq0}, we infer that $E_2=\emptyset$. This implies that $B\cap \{v_{i}v_{i}':\, v_{i}\in V(G)\}$ is an independent set, which again from~\eqref{Eu3} yields 
$$\rho_{e}^{o}(H_{G})=|B|\leq12|E_{0}|+12|E_{1}|+|B\cap \{v_{i}v_{i}':\, v_{i}\in V(G)\}|\le 12m+\alpha(G),$$
as desired. Hence, we have the equality $\rho_{e}^{o}(H_{G})=12m+\alpha(G)$ by~\eqref{eq0}.

Setting $k=12m+j$, the reduction is completed as $\rho_{e}^{o}(H_{G})\geq k$ if and only if $\alpha(G)\geq j$. Note that if $G$ is a connected cubic graph, then ($i$) $H_{G}$ is a connected bipartite graph and ($ii$) all vertices of $H_{G}$ are of even degree ($2n$ vertices are of degree 10 and the rest are of degree 2). This shows that $H_{G}$ is an Eulerian graph. Now, since {\sc Independent Set Problem} is NP-complete even for (connected) cubic graphs (see \cite{gjs}), we have the same with {\sc Edge Open Packing Problem} for Eulerian bipartite graphs. This completes the proof.
\end{proof}

In contrast to the family of graphs with maximum degree as large as possible (for which the decision problem (\ref{Uni}) is NP-complete due to Theorem~\ref{NP-complete}), the problem remains NP-complete for some special classes of sparse graphs.

\begin{theorem}\label{Comp2}
{\sc Edge Open Packing Problem} is NP-complete even when restricted to planar graphs of maximum degree at most $4$.
\end{theorem}
\begin{proof}
We only need to prove the NP-hardness of the problem. Let $G$ be a graph with $V(G)=\{v_{1},\dots,v_{n}\}$. For every $1\leq i\leq n$, consider a path $P_{i}:x_{i}z_{i}y_{i}$. Let $G'$ be obtained from $G$ by joining $z_{i}$ to $v_{i}$ for each $1\leq i\leq n$. Note that $G'$ is planar if and only if $G$ is planar. Suppose that $B'$ is a $\rho_{e}^{o}(G')$-set. Let $H$ be the subgraph of $G'$ induced by $B'\cap E(G)$. If $H$ is the null graph (that is, a graph whose vertex set and edge set are empty), then $B'\subseteq \bigcup_{i=1}^{n}E_{i}$, where $E_{i}=\{z_{i}x_{i},z_{i}y_{i},z_{i}v_{i}\}$. Set $A=\{v_{i}\in V(G):\, z_{i}v_{i}\in B'\}$. Since $B'$ is an EOP set in $G'$, it follows that $A$ is an independent set in $G$. We then get
$$\rho_{e}^{o}(G')=\big{|}B'\cap\big{(}\bigcup_{i=1}^{n}E_{i}\big{)}\big{|}=\sum_{v_{i}\in A}|B'\cap E_{i}|+\sum_{v_{i}\notin A}|B'\cap E_{i}|\leq3|A|+2(n-|A|)\leq2n+\alpha(G).$$

Suppose now that $H$ is not the null graph. By the definition of an EOP set, $H$ is the disjoint union of some induced stars $K_{1,n_{1}},\ldots,K_{1,n_{r}}$ with centers $v_{i_{1}},\ldots,v_{i_{r}}$, respectively (if one of the stars is isomorphic to $K_2$, then we consider only one of its vertices as the center). Since $B'$ is a maximum EOP set in $G'$, it necessarily follows that $E(K_{1,n_{t}})\cup \{v_{i_{t}}z_{i_{t}}\}\subseteq B'$ for all $1\leq t\leq r$. On the other hand, by the definition, $Q=\{v_{i}\in V(G):\, E_{i}\subseteq B'\}$ is an independent set in $G$. Moreover, there is no vertex in $Q$ adjacent to a vertex in $\bigcup_{i=1}^{r}(V(K_{1,n_{i}})\setminus\{v_{n_{i}}\})$. This in particular implies that $|Q|+\sum_{t=1}^{r}n_{t}\leq \alpha(G)$. We now have
$$\rho_{e}^{o}(G')=|B'|\leq3|Q|+\sum_{t=1}^{r}n_{t}+r+2(n-|Q|-\sum_{t=1}^{r}n_{t}-r)=|Q|+2n-\sum_{t=1}^{r}n_{t}-r<2n+\alpha(G).$$ 
In either case, we have proved that $\rho_{e}^{o}(G')\leq2n+\alpha(G)$. 

On the other hand, let $I$ be an $\alpha(G)$-set. We observe that $A'=(\bigcup_{v_{i}\in I}E_{i})\cup(\bigcup_{v_{i}\notin I}\{x_{i}z_{i},y_{i}z_{i}\})$ is an EOP set in $G'$. Hence, $\rho_{e}^{o}(G')\geq|A'|=2n+\alpha(G)$. This leads to the desired equality $\rho_{e}^{o}(G')=2n+\alpha(G)$

Our reduction is completed by taking $\rho_{e}^{o}(G')\geq k$ if and only if $\alpha(G)\geq j$ into account, where $k=j+2n$. Because {\sc Independent Set Problem} is NP-complete for planar graphs of maximum degree $3$ (see~\cite{gjs}), $\Delta(G')=\Delta(G)+1$, and the fact that $G'$ is planar if $G$ is planar, we have the NP-completeness of {\sc Edge Open Packing Problem} for planar graphs with maximum degree $4$.
\end{proof}


\section{An efficient algorithm for edge open packing in trees}\label{sec:treealg}

We will use the following notation throughout this section.
Let $T$ be a tree. Arbitrarily choose a vertex $r$, making $T$ a rooted tree with root $r$.
For every $v\in V(T)$, let $T_v$ denote the rooted subtree of $T$ induced by $v$, as its root, along with all descendants of $v$ in $T$. In particular, $T_r$ is isomorphic to $T$ with $r$ as its root. For every $v\in V(T)$, let $C(v)$ denote the set of children of $v$ in $T$.  

We need the following versions of optimal EOP sets in a (rooted) tree.
\begin{itemize}
\item $\roeo(T_v)=\max\{|S|:\, S\textrm{ is an EOP set in }T_v\}$;\\
(This is the standard EOP number of the subtree $T_v$. In particular, $\roeo(T_r)=\roeo(T)$.)

\item $\rho^c(T_v)= \max\{|S|:\, S\textrm{ is an EOP set in }T_v\textrm{, where }v\textrm{ is the center of a star in }T_v[S]\}$;

\item $\rho^\ell(T_v)= \max\{|S|:\, S\textrm{ is an EOP set in }T_v\textrm{, where }v\textrm{ is a leaf of a star in }T_v[S]\}$;

\item $\rho'(T_v)= \max\{|S|:\, S\textrm{ is an EOP set in }T_v\textrm{, where }v\textrm{ is not in }V(T_v[S])\}$;\\
(This is the maximum cardinality of an EOP set $S$ in $T_v$ in which $v$ is not incident with an edge of $S$.) 

\item $\rho''(T_v)= \max\{|S|:\, S\textrm{ is an EOP set in }T_v\textrm{, where }N[v]\cap V(T_v[S])=\emptyset\}$\\
(This is the maximum cardinality of an EOP set $S$ in $T_v$ such that neither $v$ nor any of its children is incident with an edge of $S$.).
\end{itemize}

Since every EOP set that appears in the definition of $\rho''(T_v)$ also satisfies the conditions from the definition of $\rho'(T_v)$, we immediately infer the following relation:
\begin{equation}\label{eq:doubleprime-prime}
\rho'(T_v)\ge \rho''(T_v)\ \mbox{for each}\ v\in V(T).
\end{equation}

We next state a basic and important remark, which follows from the fact that in an arbitrary EOP set $S$ of $T_v$ one of the following possibilities occurs: $(i)$ $v$ is the center of a star from $T_v[S]$, $(ii)$ $v$ is a leaf of a star from  $T_v[S]$, or $(iii)$ $v$ is not incident with an edge from $S$. (If $v$ belongs to a $K_{1,1}$-component of $T_v[S]$, then we may consider $v$ either as the leaf or as the center of the corresponding star $K_{1,1}$.)   

\begin{obs}\label{ob:main}
If $T$ is a (rooted) tree and $v\in V(T)$, then $$\roeo(T_v)=\max\{\rho^c(T_v),\rho^\ell(T_v),\rho'(T_v)\}.$$
\end{obs}

From the above observation, one might think that the parameter $\rho''(T_v)$ is not needed, yet it turns out that the calculations are easier if we consider it as well. Note that if $v$ is a leaf of the tree $T$, then $\roeo(T_v)=0$, and so all four parameters are also $0$. For a further simple example, consider the case when $C(v)$ consists only of leaves (and so $T_v$ is a star). Then, $\roeo(T_v)=\rho^c(T_v)=|C(v)|$, while $\rho^\ell(T_v)=1$ and $\rho'(T_v)=\rho''(T_v)=0$. 

Now we will calculate $\rho^c(T_v),\rho^\ell(T_v),\rho'(T_v)$ and $\rho''(T_v)$ by using the values of the parameters in vertices from $C(v)$. 

For calculating $\rho^c(T_v)$, we need to define two types of vertices $u$ with respect to the values of $\rho'$ and $\rho''$ of the corresponding subtrees $T_u$.
Let $u\in V(T)$. If $$\rho''(T_u)+1\ge \rho'(T_u),$$ then $u$ is of {\em Type 1}; otherwise $u$ is of {\em Type 2}.

\begin{lemma}\label{lem:rhoc}
For the rooted tree $T$ and $v\in V(T)$,
\begin{displaymath}
\rho^c(T_v)= \left\{ \begin{array}{ll}
\sum\limits_{\substack{u\in C(v),\\ u\textrm{ Type 1}}}{\hskip-10pt(\rho''(T_u)+1)}\,\, + \sum\limits_{\substack{u\in C(v),\\u\textrm{ Type 2}}}{\hskip-8pt\rho'(T_u)}& \textrm{if there is some } u\in C(v)\textrm{ of Type 1},\\ \\
\sum\limits_{u\in C(v)}{\hskip-8pt\rho'(T_u)}+\min\limits_{u\in C(v)}\{\rho''(T_u)+1-\rho'(T_u)\} & \textrm{if all }u\in C(v)\textrm{ are of Type 2}.
\end{array} \right.
\end{displaymath}
\end{lemma}
\begin{proof}
Let $S$ be an EOP set of $T_v$, which satisfies the conditions from the definition of $\rho^c(T_v)$.
That is, $v$ is the center of a star $R$ in $T_v[S]$. Therefore, at least one vertex in $C(v)$ is a leaf of the star $R$. Let $C'\subseteq C(v)$ be the set of leaves from $R$, while $C''=C(v)\setminus C'$ (it is possible that $C''=\emptyset$). For each $u\in C(v)$, let $S_u=S\cap E(T_u)$. Note that $$|S|=|C'|+\sum_{u\in C(v)}{|S_u|}.$$ Let $u\in C'$. Because $v$ is the center of the star $R$ in $T_v[S]$ and since $uv\in S$, neither $u$ nor any of the children of $u$ in $T_u$ can be incident with an edge in $S_u$. We infer that $|S_u|\le \rho''(T_u)$. On the other hand, if $u\in C''$, then $u$ cannot be incident with an edge in $S_u$, yet its children may be. Therefore, $|S_u|\le \rho'(T_u)$ in this case. 

First, assume that there exists a vertex in $C(v)$ of Type 1. Now, if $u\in C''$ is of Type 1, then $|S_u|\le \rho'(T_u)\le \rho''(T_u)+1$. Similarly, if $u\in C'$ is of Type 2, then $|S_u|\le \rho''(T_u)<\rho'(T_u)-1$. Thus, 
$$\begin{array}{lll}
|S| & = & |C'|+\hskip-10pt\sum\limits_{\substack{u\in C',\\ u\textrm{ Type 1}}}{\hskip-10pt|S_u|}+\sum\limits_{\substack{u\in C',\\ u\textrm{ Type 2}}}{\hskip-10pt|S_u|}+\sum\limits_{\substack{u\in C'',\\ u\textrm{ Type 1}}}{\hskip-10pt|S_u|}+\sum\limits_{\substack{u\in C'',\\ u\textrm{ Type 2}}}{\hskip-10pt|S_u|}\\ \\
& \le & |C'|+\hskip-10pt\sum\limits_{\substack{u\in C',\\ u\textrm{ Type 1}}}{\hskip-10pt \rho''(T_u)}+\sum\limits_{\substack{u\in C',\\ u\textrm{ Type 2}}}{\hskip-10pt (\rho'(T_u)-2)}+\sum\limits_{\substack{u\in C'',\\ u\textrm{ Type 1}}}{\hskip-10pt(\rho''(T_u)+1)}+\sum\limits_{\substack{u\in C'',\\ u\textrm{ Type 2}}}{\hskip-10pt\rho'(T_u)}\\ \\
& = & \hskip-10pt\sum\limits_{\substack{u\in C',\\ u\textrm{ Type 1}}}{\hskip-10pt (\rho''(T_u)+1)}+\sum\limits_{\substack{u\in C',\\ u\textrm{ Type 2}}}{\hskip-10pt (\rho'(T_u)-1)}+\sum\limits_{\substack{u\in C'',\\ u\textrm{ Type 1}}}{\hskip-10pt(\rho''(T_u)+1)}+\sum\limits_{\substack{u\in C'',\\ u\textrm{ Type 2}}}{\hskip-10pt\rho'(T_u)}.\\
\end{array}
$$
and the above expression is in turn less than of equal to $$\sum\limits_{\substack{u\in C(v),\\ u\textrm{ Type 1}}}{\hskip-10pt(\rho''(T_u)+1)}\,\, + \sum\limits_{\substack{u\in C(v),\\u\textrm{ Type 2}}}{\hskip-8pt\rho'(T_u)}.$$

Second, assume that all vertices in $C(v)$ are of Type 2. Then, 
$$\begin{array}{lll}
|S| &\le& |C'|+\hskip-6pt\sum\limits_{\substack{u\in C',\\ u\textrm{ Type 2}}}{\hskip-10pt \rho''(T_u)}+\sum\limits_{\substack{u\in C'',\\ u\textrm{ Type 2}}}{\hskip-10pt\rho'(T_u)}\\ \\
&=& \hskip-6pt\sum\limits_{\substack{u\in C',\\ u\textrm{ Type 2}}}{\hskip-10pt (\rho''(T_u)+1)}+\sum\limits_{\substack{u\in C'',\\ u\textrm{ Type 2}}}{\hskip-10pt\rho'(T_u)}\\ \\
&\le& \sum\limits_{u\in C(v)}{\hskip-8pt\rho'(T_u)}+\min\limits_{u\in C(v)}\{\rho''(T_u)+1-\rho'(T_u)\},
\end{array}$$
where the latter inequality holds because $\rho'(T_u)>\rho''(T_u)+1$ for all $u\in C(v)$ and since $|C'|>0$. All in all, we have proved that $\rho^c(T_v)$ cannot be larger than the expression on the right side of the equality stated in the lemma.

For the reverse inequalities, we also consider two cases. Suppose that there exists $u\in C(v)$ such that $\rho''(T_u)+1\ge \rho'(T_u)$ (that is, $u$ is of Type 1). Let $B$ be an EOP set of $T_v$, where $v$ is the center of a star in $T_{v}[B]$, such that $vu\in B$ if and only if $u$ is of Type 1, while $T_u[B]$ depends on the type of $u\in C(v)$. That is, if $u$ is of Type 1, then let $E(T_u[B])$ be a $\rho''(T_u)$-set, and if $u$ is of Type 2, then let $E(T_u[B])$ be a $\rho'(T_u)$-set. 
Therefore, the inequality 
$$\rho^c(T_v)\geq|B|=\sum\limits_{\substack{u\in C(v),\\ u\textrm{ Type 1}}}{\hskip-10pt(\rho''(T_u)+1)}\,\, + \sum\limits_{\substack{u\in C(v),\\u\textrm{ Type 2}}}{\hskip-8pt\rho'(T_u)}$$
readily follows. On the other hand, if all children of $v$ are of Type 2, then let $u^{*}$ be a vertex in $C(v)$ such that $\rho''(T_{u^{*}})-\rho'(T_{u^{*}})$ is the smallest possible among all vertices in $C(v)$. Letting $S^{*}$ be an EOP in $T_{v}$, in which $v$ is the center of the star $R$ containing only $v$ and $u^{*}$, we infer
$$\begin{array}{lll}
\rho^c(T_v)\geq|S^{*}|&\geq&\sum\limits_{u\in C(v)\setminus\{u^*\}}{\hskip-8pt\rho'(T_u)+\rho''(T_{u^{*}})+1}\\ \\
&=& \sum\limits_{u\in C(v)}{\hskip-8pt\rho'(T_u)}+\rho''(T_{u^{*}})+1-\rho'(T_{u^*})\\ \\
&=&\sum\limits_{u\in C(v)}{\hskip-8pt\rho'(T_u)}+\min\limits_{u\in C(v)}\{\rho''(T_u)+1-\rho'(T_u)\},
\end{array} $$
leading to the desired equality. The proof is complete.
\end{proof}

\begin{lemma}\label{lem:rhoell}
For the rooted tree $T$ and $v\in V(T)$,
\begin{displaymath}
\rho^\ell(T_v)= \max\limits_{u\in C(v)}\Big\{\!\max\{\rho^c(T_u),\rho''(T_u)\}+1+\hskip-8pt\sum\limits_{w\in C(v)\setminus\{u\}}\hskip-10pt{\rho'(T_w)}\,\Big\}.
\end{displaymath}
\end{lemma}
\begin{proof}
Let $S$ be an EOP set of $T_v$, which satisfies the conditions from the definition of $\rho^\ell(T_v)$. That is, $v$ is a leaf of a star $R$ in $T_v[S]$. Therefore, exactly one vertex in $C(v)$ is the center of the star $R$, and denote it by $u$. Clearly, none of the edges $vw$, where $w\in C(v)\setminus\{u\}$, belongs to $S$. In addition, since $S$ is an EOP set of $T_v$, no vertex $w\in C(v)\setminus\{u\}$ is incident with an edge from $S$. For any $x\in C(v)$, let $S_x=S\cap E(T_x)$. If $w\ne u$, then $S_w$ is an EOP set of $T_w$ such that $w\notin V(T_w[S_w])$. Therefore, $|S_w|\le \rho'(T_w)$. On the other hand, since $u$ is the center of the star $R$ in $T_v[S]$, there are two possibilities for the edges from $S_u$ depending on the order of $R$. If $R\cong K_{1,1}$, then none of the vertices in $N[u]\cap V(T_u)$ can be in $T_u[S_u]$, and so $|S_u|\le \rho''(T_u)$. If $|V(R)|\ge 3$, then some of the edges between $u$ and vertices in $C(u)$ are in $S_u$. Hence, $|S_u|\le \rho^c(T_u)$. Taking into account both possibilities, we readily infer that $|S_u|\le \max\{\rho''(T_u), \rho^c(T_u)\}.$ Combining all findings from above, we get
$$|S|=|S_u|+1+\hskip-10pt\sum\limits_{w\in C(v)\setminus\{u\}}{\hskip-10pt|S_w|}\le \max\limits_{u\in C(v)}\Big\{\!\max\{\rho^c(T_u),\rho''(T_u)\}+1+\hskip-8pt\sum\limits_{w\in C(v)\setminus\{u\}}\hskip-10pt{\rho'(T_w)}\,\Big\}.$$ 

On the other hand, the construction described in the previous paragraph leads to an EOP set $S$ of $T_v$ of desired cardinality. In this construction, $v$ is a leaf of a star $R$ in $T_v[S]$, one of its children $u$ is the center of $R$, $S\cap E(T_u)$ is either a $\rho^c(T_u)$-set or a $\rho''(T_u)$-set (depending on which of these sets is larger), and in all other sets $S\cap E(T_w)$, where $w\in C(v)\setminus\{u\}$, we can take a $\rho'(T_w)$-set. By taking a maximum possible such set $S$ over all  $u\in C(v)$, we infer that the claimed value for $|S|$ is achieved.
\end{proof}

\begin{lemma}\label{lem:rhoprime}
For the rooted tree $T$ and $v\in V(T)$,
\begin{displaymath}
\rho'(T_v)= \sum\limits_{u\in C(v)}{\roeo(T_u)}.
\end{displaymath}
\end{lemma}
\begin{proof}
Let $S$ be an EOP set of $T_v$, which satisfies the conditions from the definition of $\rho'(T_v)$. That is, $v$ is not incident with an edge from $S$. For any $u\in C(v)$, let $S_u=S\cap E(T_u)$. Note that $S_u$ is an EOP set of $T_u$, and so $|S_u|\le \roeo(T_u)$.
Thus, $|S|=\sum_{u\in C(v)}{|S_u|}\le \sum_{u\in C(v)}{\roeo(T_u)}$. On the other hand, taking a $\roeo(T_u)$-set $S_u$ in $T_u$ for each $u\in C(v)$, the union of all these sets $S_u$ yields an EOP set of $T_v$ of cardinality $\sum_{u\in C(v)}{\roeo(T_u)}$, which completes the proof of the lemma. 
\end{proof}

\begin{algorithm}[ht!]
{\caption{EOP number of a tree}
\label{al:tree}
\KwIn{A tree $T$ on $n$ vertices.}
\KwOut{$\roeo(T)$.}

\BlankLine
{
    Fix a root $r\in V(T)$, and let $v_1,v_2,\ldots, v_{n-1},v_n=r$ be the vertices of $T$ listed
    in reverse order with respect to the time they are visited by a breadth-first traversal of $T$ from $r$; 

    \For{$i = 1, \ldots, n$}
    {
        Let $C(v_i)$ be the set of children of $v_i$\\

        \If{$C(v_i)= \emptyset$}
        {
            $\roeo(T_{v_i})\leftarrow 0$, $\rho^c(T_{v_i})\leftarrow 0$, $\rho^\ell(T_{v_i})\leftarrow 0$, $\rho'(T_{v_i})\leftarrow 0$, $\rho''(T_{v_i})\leftarrow 0$; $v_i\leftarrow$\,Type 1
        }
        \Else
        {   $\rho''(T_{v_i}) \leftarrow \sum\limits_{u\in C(v_{i})}{\rho'(T_u)}$
        
           $\rho'(T_{v_i}) \leftarrow \sum\limits_{u\in C(v_{i})}{\roeo(T_u)}$
        
            \If{$\rho''(T_{v_i})+1\ge \rho'(T_{v_i})$}
            {$v_i\leftarrow$\,Type 1}
                       \Else 
            {$v_i\leftarrow$\,Type 2}
        
          \For{$u\in C(v_i)$} 
           {
           $M_u \leftarrow \max\{\rho^c(T_u),\rho''(T_u)\}+1+\hskip-8pt\sum\limits_{w\in C(v_{i})\setminus\{u\}}\hskip-10pt{\rho'(T_w)}$
                        }
            
           $\rho^\ell(T_{v_i}) \leftarrow \max\limits_{u\in C(v_{i})}{M_u}$

          \If{$\forall u\in C(v_{i})[u={\rm Type\,\, 2}]$}
            {    $D\leftarrow \min\limits_{u\in C(v_{i})}\{\rho''(T_u)+1-\rho'(T_u)\}$

                $\rho^c(T_{v_i}) \leftarrow \sum\limits_{u\in C(v_{i})}{\hskip-8pt\rho'(T_u)}+D$

            }
            \Else
          {
                $\rho^c(T_{v_i})\leftarrow\sum\limits_{\substack{u\in C(v_{i}),\\ u=\textrm{ Type 1}}}{\hskip-10pt(\rho''(T_u)+1)}\,\, + \sum\limits_{\substack{u\in C(v_{i}),\\u=\textrm{ Type 2}}}{\hskip-8pt\rho'(T_u)}$
            }
            
            $\roeo(T_{v_i})\leftarrow\max\{\rho^c(T_{v_i}),\rho^\ell(T_{v_i}),\rho'(T_{v_i})\}$

        }
    }
    \Return{$\roeo(T_{r})$}
}
}
\end{algorithm}

\begin{lemma}\label{lem:rhodoubleprime}
For the rooted tree $T$ and $v\in V(T)$,
\begin{displaymath}
\rho''(T_v)= \sum\limits_{u\in C(v)}{\rho'(T_u)}.
\end{displaymath}
\end{lemma}
\begin{proof}
Let $S$ be an EOP set of $T_v$, which satisfies the conditions from the definition of $\rho''(T_v)$. That is, $v$ is not incident with an edge from $S$, and also none of the children $u\in C(v)$ is incident with an edge from $S$. For any $u\in C(v)$, let $S_u=S\cap E(T_u)$. Note that $S_u$ satisfies the conditions from the definition of $\rho'(T_u)$, and so $|S_u|\le \rho'(T_u)$. We infer that $|S|=\sum_{u\in C(v)}{|S_u|}\le \sum_{u\in C(v)}{\rho'(T_u)}$. On the other hand, taking a largest EOP set $S_u$ in $T_u$ for each $u\in C(v)$, such that $u$ is not incident with an edge from $S_u$, the union $S'$ of all these sets $S_u$ yields an EOP set of $T_v$ of cardinality $\sum_{u\in C(v)}{\rho'(T_u)}$. Clearly, $S'$ is an EOP set of $T_v$, such that $N[v]\cap V(T_v[S'])=\emptyset$, which completes the proof of the lemma. 
\end{proof} 

Based on the above statements we deduce an efficient algorithm for determining $\roeo(T)$ for an arbitrary tree $T$ (see Algorithm~\ref{al:tree}). 

As a consequence of Observation~\ref{ob:main} and Lemmas~\ref{lem:rhoc},~\ref{lem:rhoell},~\ref{lem:rhoprime} and~\ref{lem:rhodoubleprime}, we derive that Algorithm~\ref{al:tree} correctly computes the EOP number of a tree $T$. 

Let us now consider the algorithm's complexity. (Let $n$ and $m$ stand for the order and the size of $T$, respectively). The vertex ordering $v_1,\ldots,v_n$ can be computed in $O(n)$ time using a breadth-first traversal from $r$. At each leaf (vertex $v_i$ with $C(v_i) = \emptyset$), a constant number of operations is performed. At each internal vertex $v_i$,
$O(|C(v_i)|)$ operations are performed (assuming, as usual, that adding and comparing two numbers can be done in $O(1)$ time). This clearly yields that one can bound the time complexity of the algorithm by $O(m)$, which is equivalent to $O(n)$. We have thus proved the following result. 

\begin{theorem}\label{thm:trees}
There exists a linear-time algorithm for computing the edge open packing number of a tree.
\end{theorem}

We note that the results of this section and their proofs imply that Algorithm~\ref{al:tree} can be modified so that it also obtains a maximum EOP set of the input tree $T$. At every vertex $v_i$, we can additionally obtain the edges of $S_{v_i}$ according to the corresponding parameter. Notably, in lines 7, 8, 15, 18 or 20, and 21, we can additionally obtain the edges of the corresponding five copies of $T_{v_i}$ by using the constructions from corresponding lemmas. 


\section{Two bounds and realization results}
\label{sec:twobounds}

In this section, we first describe the structure of all graphs with given EOP number and minimum degree for which the number of edges is as small as possible. We then study how the EOP number changes under edge removal.

\subsection{Extremal graphs for the bound $\roeo(G)\le |E(G)|/\delta(G)$  }

Chelladurai et al.~\cite{cks} showed that the EOP number of a graph $G$ of size $m$ can be bounded from above by $m/\delta(G)$. They also posed the open problem of characterizing all graphs for which the equality holds in the upper bound. 

In this section, we give a complete solution to the above-mentioned problem. To do so, we define the family $\cal F$ as follows. Let $G$ be a bipartite graph of minimum degree $k\geq2$ with partite sets $A\cup C$ and $B$ such that
every vertex in $B$ has $1$ and $k-1$ neighbors in $A$ and $C$, respectively.

\begin{theorem}\label{Char}
For any graph $G$ of size $m$, $\rho_{e}^{o}(G)=m/\delta(G)$ if and only if either $G$ is a disjoint union of stars or $G\in \cal F$.
\end{theorem}
\begin{proof}
The equality evidently holds when $G$ is isomorphic to a disjoint union of stars. So, let $G\in \cal F$. We observe that the subgraph of $G$ induced by $A\cup B$ is a disjoint union of induced stars. Therefore, $[A,B]$ is an EOP set in $G$. By the structure of $G$, described in the definition $\cal F$, we observe that 
$$\rho_{e}^{o}(G)\geq|[A,B]|=|B|=\dfrac{|B|+(k-1)|B|}{k}=\dfrac{m}{\delta(G)}.$$
This results in the equality $\rho_{e}^{o}(G)=m/\delta(G)$.

In order to prove the converse implication, we first give a proof of the upper bound $\rho_{e}^{o}(G)\leq m/\delta(G)$ for any graph $G$ of size $m$. Let $P$ be a $\rho_{e}^{o}(G)$-set and let $xy$ be an arbitrary edge in $P$. Since $P$ is an EOP set of $G$, it follows that at least one of $x$ and $y$, say $x$, is adjacent to $\deg(x)-1$ vertices in $V(G)\setminus V(G[P])$. This shows that every edge in $P$ is adjacent to at least $\delta(G)-1$ edges not in $P$. Therefore,
\begin{equation}\label{EQ1}
m=|P|+|[V(G[P]),V(G)\setminus V(G[P])]|+|[V(G)\setminus V(G[P]),V(G)\setminus V(G[P])]|\geq|P|+(\delta(G)-1)|P|.
\end{equation}
So, $\rho_{e}^{o}(G)=|P|\leq m/\delta(G)$.

Now let $\rho_{e}^{o}(G)=m/\delta(G)$ for a graph $G$, of size $m$ and order $n$, different from a disjoint union of stars. Therefore, we necessarily have equality in (\ref{EQ1}). With this in mind, there is no edge with both endvertices in $V(G)\setminus V(G[V(P)])$, for otherwise we end up with $\rho_{e}^{o}(G)<m/\delta(G)$. Note that $G[V(P)]$ is a disjoint union of some induced stars $S_{1},\ldots,S_{t}$ with centers $a_{1},\ldots,a_{t}$, respectively. Set $A'=\{a_{1},\ldots,a_{t}\}$, $B'=\bigcup_{i=1}^{t}(V(S_{i})\setminus \{a_{i}\})$ and $C'=V(G)\setminus(A'\cup B')=V(G)\setminus V(G[V(P)])$. Notice that $C'\neq \emptyset$ because $G$ is different from a disjoint union of induced stars. Together the equality in (\ref{EQ1}) and the fact that $P$ is an EOP set in $G$ imply that every vertex in $B'$ is adjacent to precisely $\delta(G)-1$ vertices in $C'$ and to exactly one vertex in $A'$ (in particular, this shows that $\delta(G)\geq2$). Moreover, no vertex in $A'$ has a neighbor in $A'\cup C'$. In particular, the induced subgraph $G[A'\cup C']$ is edgeless. On the other hand, since $P$ is an EOP set in $G$, there is no edge with endvertices in $B'$. The above argument shows that $G[B'\cup C']$ and $G[A'\cup B']$ are bipartite graphs in which every vertex in $B'$ is adjacent to $\delta(G)-1$ vertices in $C'$ and to one vertex in $A'$. We now observe that $\delta(G)$, $A'$, $B'$ and $C'$ have the same roles as $k$, $A$, $B$ and $C$ have in the description of the members of $\cal F$, respectively. Consequently, $G$ belongs to $\cal F$. 
\end{proof}

\subsection{Effect of edge removal}

Analysis of the effect of edge removal on the value of $\rho_{e}^{o}$ was suggested in~\cite{cks}, and here we give a comprehensive response. Note that characterizations of the graphs with $\rho_{e}^{o}(G)\in \{1,2\}$ were also given in~\cite{cks}. In fact, it is easy to see that $\rho_{e}^{o}(G)=1$ if and only if $G\cong K_{n}$ for $n\geq2$. Moreover, $\rho_{e}^{o}(G)=2$ if and only if $(i)$ $2\leq \mbox{diam}(G)\leq4$, $(ii)$ $G$ is claw-free and $(iii)$ for any two nonadjacent edges $e_{1}=uv$ and $e_{2}=xy$ having no common edge, every vertex in $V(G)\setminus \{u,v,x,y\}$ is adjacent to at least two vertices in $\{u,v,x,y\}$.

It is clear that $\rho_{e}^{o}(K_{n}-e)=2=\rho_{e}^{o}(K_{n})+1$ for each edge $e$ when $n\geq3$. So, we consider the graphs $G$ with $\rho_{e}^{o}(G)\geq2$. 

\begin{theorem}\label{removal}
For any graph $G$ of order $n\geq3$ and $e\in E(G)$, the following statements hold.
\begin{enumerate}[(i)]
    \item If $\rho_{e}^{o}(G)=2$, then $1\leq \rho_{e}^{o}(G-e)\leq3$.
    \item  If $\rho_{e}^{o}(G)\geq3$, then $\rho_{e}^{o}(G)-1\leq \rho_{e}^{o}(G-e)\leq2(\rho_{e}^{o}(G)-1)$.
\end{enumerate}
\end{theorem}
\begin{proof}
First, let us consider the upper bound.  Let $B$ be a $\rho_{e}^{o}(G-e)$-set. Note that the subgraph of $G$ induced by $B$ is a disjoint union of induced stars $S_{1},\ldots,S_{k}$ with centers $s_{1},\ldots,s_{k}$, respectively (if $S_{i}$ has only one edge for some $1\leq i\leq k$, then we only consider one of its endvertices as the center). If $B$ turns out to be an EOP set of $G$, then $\rho_{e}^{o}(G-e)\leq \rho_{e}^{o}(G)$, which is less than or equal to $2(\roeo(G)-1)$. So, we may assume that $B$ is not an EOP set of $G$. This implies that $e$ joins a vertex $x\in V(S_{i})$ to a vertex $y\in V(S_{j})$ for some $1\leq i \le j\leq k$. We deal with the following possibilities.\vspace{0.75mm}\\
(a) Let $x=s_{i}$ and $y$ is a leaf of $S_{j}$ for some $1\leq i< j\leq k$. Then, $B\setminus \{ys_{j}\}$ is an EOP set of $G$ of cardinality $\rho_{e}^{o}(G-e)-1$. So, $\rho_{e}^{o}(G-e)\leq \rho_{e}^{o}(G)+1$.\vspace{0.75mm}\\
(b) Let $x$ be a leaf of $S_{i}$ and $y$ be a leaf of $S_{j}$ for some $1\leq i\le j\leq k$ ($i$ may be equal to $j$). Then, $B\setminus\{xs_{i}\}$ is an EOP set of $G$ of cardinality $\rho_{e}^{o}(G-e)-1$. Hence, $\rho_{e}^{o}(G-e)\leq \rho_{e}^{o}(G)+1$.\vspace{0.75mm}\\
(c) Let $x=s_{i}$ and $y=s_{j}$ for some $1\leq i< j\leq k$. We may assume, without loss of generality, that $|E(S_{i})|\leq|E(S_{j})|$. In such a case, $B'=\big{(}B\setminus E(S_{i})\big{)}\cup \{e\}$ is an EOP set in $G$. Thus, $\rho_{e}^{o}(G)\geq|B'|\geq|B|/2+1=\rho_{e}^{o}(G-e)/2+1$.\vspace{0.75mm}\\
Summing up, we have proved that $\rho_{e}^{o}(G-e)\leq \max\{2\rho_{e}^{o}(G)-2,\rho_{e}^{o}(G)+1\}$ in each case, which proves the upper bounds in $(i)$ and $(ii)$. 

For the lower bounds, note that if $B$ is a $\rho_{e}^{o}(G)$-set, then $B\setminus \{e\}$ is an EOP set in $G-e$ for each $e\in E(G)$. This leads to $\rho_{e}^{o}(G)-1\leq \rho_{e}^{o}(G-e)$.
\end{proof}

In the next remark we show that all integer values between the lower and upper bounds in Theorem~\ref{removal} are realizable. 
\begin{obs}
For any integers $b\geq3$ and $a$ with $b-1\leq a\leq2b-2$, there exists a connected graph $G$ and $e\in E(G)$ such that $\rho_{e}^{o}(G)=b$ and $\rho_{e}^{o}(G-e)=a$. In addition, for any $a\in\{1,2,3\}$ there exists a connected graph $G$ and $e\in E(G)$ such that $\roeo(G)=2$ and $\roeo(G-e)=a$. 
\end{obs}
\begin{proof}
Let $b\ge 3$. First, let $b+1\leq a\leq2b-2$. Let $G$ be obtained from $K_{1,b+1}$, with leaves $v_{1},\ldots,v_{b+1}$ and central vertex $v$, by $(i)$ attaching $\lfloor a/2\rfloor$ new vertices to $v_{1}$ as leaves and also $\lceil a/2\rceil$ new vertices to $v_{2}$ as leaves, and $(ii)$ adding the edge $v_{1}v_{2}$. It is not difficult to check that $\rho_{e}^{o}(G)=b$ and $\rho_{e}^{o}(G-v_{1}v_{2})=a$.
Second, let $a=b$. Consider the graph $G'$ obtained from the above-mentioned graph $G$ by removing the vertices in $(i)$. We then have $\rho_{e}^{o}(G')=\rho_{e}^{o}(G'-vv_{1})=b$. Third, let $a=b-1$. Let $G''$ be obtained from the paths $vv_{i}w_{i}$, for $1\leq i\leq b$, by adding edges in such a way that the set $\{w_{1},w_{2},\ldots,w_{b}\}$ induces a clique. One can easily check that $\rho_{e}^{o}(G'')=b$ and $\rho_{e}^{o}(G''-vv_{b})=b-1$.

Suppose now that $b=2$. Let $G$ be obtained from $K_{n}$, with $n\geq3$, by joining a new vertex $v$ to a vertex $u\in V(K_{n})$. It is then clear that $\rho_{e}^{o}(G)=2=\rho_{e}^{o}(G-uv)+1$. Moreover, $\rho_{e}^{o}(G)=2=\rho_{e}^{o}(G-xy)-1$ for any edge $xy$ of $K_{n}$ with $x,y\neq u$. Finally, $\rho_{e}^{o}(G)=2=\rho_{e}^{o}(G-uw)$ in which $u\neq w\in V(K_{n})$.
\end{proof}

\section*{Acknowledgement}
B.B. was supported by the Slovenian Research and Innovation Agency (ARIS) under the grants P1-0297, N1-0285, J1-3002, and J1-4008. 



\begin{thebibliography}{99}

\bibitem {Cam} K.~Cameron, \textit{Induced matchings}, Discrete Appl. Math. {\bf 24} (1989), 97--102.
\bibitem {cccd} D.M.~Cardoso, J.O.~Cerdeira, J.P.~Cruz and C.~Dominic, \textit{Injective edge coloring of graphs}, Filomat, {\bf 33} (2019), 6411--6423.
\bibitem {cks} G.~Chelladurai, K.~Kalimuthu and S.~Soundararajan, {\em Edge open packing sets in graphs}, RAIRO-Oper. Res. {\bf 56} (2022), 3765--3776.
\bibitem {ddl} K.K.~Dabrowski, M.~Demange and V.V.~Lozin, {\em New results on maximum induced matchings in bipartite graphs and beyond}, Theoret. Comput. Sci. {\bf 478} (2013), 33--40.
\bibitem {fhl} F.~Foucaud, H.~Hocquard and D.~Lajou, {\em Complexity and algorithms for injective edge-coloring in graphs}, Inform. Process. Lett. {\bf 170} (2021), 106121. 
\bibitem {gj} M.R.~Garey and D.S.~Johnson, Computers and intractability: A guide to the theory of NP-completeness, W.H. Freeman $\&$ Co., New York, USA, 1979.
\bibitem {gjs} M.R.~Garey, D.S.~Johnson and L.~Stockmeyer, {\em Some simplified NP-complete graph problems}, Theoret. Comput. Sci. {\bf 1} (1976), 237--267.
\bibitem {Lozin} V.V.~Lozin, {\em On maximum induced matchings in bipartite graphs}, Inform. Process. Lett. {\bf 81} (2002), 7--11.
\bibitem {msy} Z.~Miao, Y.~Song and G.~Yu, {\em Note on injective edge-coloring of graphs}, Discrete Appl. Math. {\bf 310} (2022), 65--74.
\bibitem {sv} L.J.~Stockmeyer and V.V. Vazirani, {\em NP-completeness of some generalizations of the maximum matching problem}, Inform. Process. Lett. {\bf 15} (1982), 14--19.
\bibitem {West} D.B.~West, Introduction to Graph Theory (Second Edition), Prentice Hall, USA, 2001.
\bibitem{zito} M.~Zito, {\em Induced matchings in regular graphs and trees}, Graph-theoretic concepts in computer science, {\bf 1665} (1999), 89--101.
\end{thebibliography}
\end{document}